\documentclass[11pt]{amsart}

\usepackage{amssymb}
\usepackage{latexsym}
\input xypic
\xyoption {all}

\usepackage{color}

\theoremstyle{plain}
\newtheorem{thm}{Theorem}[section]

\newtheorem{prop}[thm]{Proposition}

\newtheorem{remark}[thm]{Remark}

\newtheorem{lemma}[thm]{Lemma}

\newtheorem{prob}[thm]{Problem}

\newtheorem{defn}[thm]{Definition}

\newtheorem{ex}[thm]{Example}

\theoremstyle{definition}
\newtheorem*{ack}{Acknowledgements}

\numberwithin{equation}{section}

\newcommand{\supp}{\operatorname{supp}}

\newcommand{\sgn}{\operatorname{sgn}}
\newcommand{\lb}{\label}

\begin{document}

\dedicatory{Dedicated to the memory of Joram Lindenstrauss}

\title{Narrow and $\ell_2$-strictly singular operators from $L_p$}

%----------Author 1

\author{V.~Mykhaylyuk}

\address{Department of Mathematics\\
Chernivtsi National University\\ str. Kotsjubyns'koho 2,
Chernivtsi, 58012 Ukraine}

\email{vmykhaylyuk@ukr.net}

\author{M.~Popov}

\address{Department of Mathematics\\
Chernivtsi National University\\ str. Kotsjubyns'koho 2,
Chernivtsi, 58012 Ukraine}

\email{\texttt{misham.popov@gmail.com}}

\author{B.~Randrianantoanina}

\address{Department of Mathematics\\
Miami University\\
Oxford, OH, 45056, USA}

\email{\texttt{randrib@muohio.edu}}

\author{G.~Schechtman}\thanks{G.S. was supported by the Israel Science Foundation and by the U.S.-Israel Binational Science Foundation.}

\address{Department of Mathematics\\
Weizmann Institute of Science\\
Rehovot, Israel 76100 }

\email{\texttt{gideon@weizmann.ac.il}}



\begin{abstract}
In the first part of the paper we prove that for $2 < p, r < \infty$ every operator $T: L_p \to \ell_r$ is narrow. This completes the list of   sequence and function Lebesgue spaces $X$ with the property that every operator $T:L_p \to X$ is narrow.

 Next, using similar methods we prove that every $\ell_2$-strictly singular operator from $L_p$, $1<p<\infty$, to any Banach space with an unconditional basis, is narrow, which partially answers a question of Plichko and Popov posed in 1990.

A theorem of H.~P.~Rosenthal asserts that if  an operator $T$ on $L_1[0,1]$ satisfies the
assumption that for each measurable set $A \subseteq [0,1]$ the restriction $T \bigl|_{L_1(A)}$ is not an isomorphic embedding, then $T$ is narrow.
 (Here  $L_1(A) = \{x \in L_1:  {\rm supp} \, x \subseteq A\}$.) Inspired by this result, in the last part of the paper, we find a sufficient condition, of a different flavor than being $\ell_2$-strictly singular, for operators on $L_p[0,1]$, $1<p<2$, to be narrow. We define a notion of a ``gentle'' growth of a function and we prove that  for $1 < p < 2$ every operator $T$ on $L_p$ which, for every $A\subseteq[0,1]$, sends a function of ``gentle" growth supported on $A$ to a function of arbitrarily small norm is narrow.
\end{abstract}

\maketitle

\section{Introduction}
\label{sec0}

In this paper
we study narrow operators on the real spaces $L_p$, for $1\le p <\infty$ (by $L_p$ we mean the $L_p[0,1]$ with the Lebesgue  measure $\mu$). Let $X$ be a Banach space. We say that a linear operator $T : L_p \to X$ is narrow if for every $\varepsilon>0$ and every measurable set $A\subseteq [0,1]$ there exists $x\in L_p$ with $x^2 = \mathbf{1}_A$, $\displaystyle{\int_{[0,1]} x \, d \mu = 0}$, so that $\|Tx\|<\varepsilon$. By $\Sigma$ we denote the $\sigma$-algebra of Lebesgue measurable subsets of $[0,1]$, and set $\Sigma^+ = \{A \in \Sigma: \, \mu(A) > 0\}$.

As is easy to see, every compact operator $T:L_p\to X$ is narrow. However  in general  the class of narrow operators is much larger than that of compact operators. All  spaces  $X$  that admit a non-compact operator from $L_p$ to $X$ also admit a narrow non-compact operator from $L_p$ to $X$ \cite[Corollary~4.19]{PR}.

In Section~\ref{seceverynarrow} we start to study the question for what Banach spaces $X$, every operator $T:L_p\to X$ is narrow and settle it for $X$ being a (function or sequence) Lebesgue space.
It is known that every operator $T:L_p\to L_r$ is narrow when $1 \leq p < 2$ and $p < r$ \cite{KP92}. This is not true for any other values of $p$ and $r$, as Example~\ref{R:Sch} below shows.

Note that, when $1 \leq r < 2$ and $p$ is arbitrary ($1 \leq p < \infty$) then by Khintchine's inequality and the fact that every operator from $\ell_2$ to $\ell_r$ is compact, every operator $T:L_p\to \ell_r$ is narrow, see \cite[p.~63]{PP}. This also holds when $1 \leq p < 2$ and $p < r$, as a consequence of the above mentioned result that every operator $T:L_p\to L_r$ is narrow.

In Section~\ref{seceverynarrow} we settle the remaining cases of $p$ and $r$ and deduce

\vspace{0,2 cm}

\noindent{\textbf{Theorem A.}} \textit{If $1\le p < \infty$ and $1\le r\not=2 < \infty$ or $1\le p<2$ and $r=2$ then every operator $T :L_p\to \ell_r$ is narrow.}

\vspace{0,2 cm}

For $2 < p \leq r < \infty$, the operator $S_{p,r}$ that sends the normalized Haar system in $L_p$ to the unit vector basis of $\ell_r$ is bounded. The proof that this operator is narrow is the main ingredient of the proof of Theorem~A.

Recall that for Banach spaces $X$, $Y$, and $Z$, an operator $T:X\to Y$ is said to be {\it $Z$-strictly singular}
if it is not an isomorphism when restricted to any isomorphic copy of $Z$ in $X$. In 1990 Plichko and Popov  \cite{PP} asked whether every $\ell_2$-strictly singular operator $T:L_p\to X$ is necessarily narrow. In Section~\ref{secstrsng}, generalizing Theorem A, we answer this question positively for Banach spaces $X$  with an unconditional basis.

\vspace{0,2 cm}

\noindent{\textbf{Theorem B.}}  \textit{For every $p$ with $1 < p < \infty$,   and every Banach space $X$ with an unconditional basis,  every  $\ell_2$-strictly singular operator $T :L_p\to X$ is narrow.}

The proof of Theorem B is quite a simple modification of that of Theorem A.
We preferred however to give the two proofs separately for the benefit of the reader who is interested only in the more concrete cases of Lebesgue spaces.

%\vspace{0,2 cm}

The motivation for the last part of the paper, Section {\ref{secGnarrow}, is the study of the relationship between narrow operators and Enflo operators on $L_p$,  that is operators so that there exists a subspace $X\subseteq L_p$ isomorphic to $L_p$ such that $T$ restricted to $X$ is an into isomorphism. Note that there are narrow operators $T$ on $L_p$ which are Enflo \cite[p.~55]{PP}.
Indeed the conditional expectation from $L_p([0,1]^2)$ onto the subspace of all functions depending on the first variable only is such an operator. A natural question that arises here and that has been studied by several authors is whether non-narrow operators on $L_p$  are always Enflo.

This is evidently true for $p=2$, but false for $p > 2$  due to the following example:

 \begin{ex} \lb{R:Sch}
Let $p>2$ and  $T = S \circ J$ where $J: L_p \to L_2$ is the inclusion embedding and $S: L_2 \to L_p$ is an isomorphic embedding. Then $T$ is not narrow and not Enflo. (Also, for all $p>r\ge 1$, the inclusion $J:L_p\to L_r$ is not narrow.)
 \end{ex}

For $p = 1$, the answer  is affirmative, which follows from the results of  Enflo and Starbird  \cite{ES}. The following remarkable result of Rosenthal (the equivalence of (c) and (d) in Theorem 1.5 of \cite{Ros84}) gives much more: necessary and sufficient conditions on an operator $T \in \mathcal L(L_1)$ to be narrow.

\begin{thm}[H.~Rosenthal] \label{Ros}
An operator $T:L_1\to L_1$ is narrow if and only if, for each measurable set $A \subseteq [0,1]$ the restriction $T \bigl|_{L_1(A)}$ is not an isomorphic embedding.
\end{thm}

What happens for $1<p<2$? This is related to the following theorem of Johnson, Maurey, Schechtman and Tzafriri  \cite{JMST}:

\begin{thm} \label{comp}
Let $1<p<\infty$ and  $T:L_p\to L_p$ be an into isomorphism. Then there exists a subspace $Y\subseteq L_p$ isomorphic to $L_p$ and so that $TY$ is isomorphic to $L_p$ and complemented in $L_p$.
\end{thm}

In fact Theorem~\ref{comp} is valid for a much wider class of spaces than $L_p$. However in the special case of $L_p$ for $1<p<2$ the proof gives more: It is enough to assume that $T$ is non-narrow to get the same conclusion (this is mentioned without proof in \cite[Theorem 4.12, item 2, p.~54]{Bou81-2}). Recently Dosev, Johnson and Schechtman \cite{DJS} proved the following strengthening of Theorem~\ref{comp} and the above mentioned variant of it.

\begin{thm} \label{DJS}
For each $1 < p < 2$ there is a constant $K_p$ such that if $T \in \mathcal L(L_p)$ is a non-narrow operator (and in particular, if it is an isomorphism) then there is a $K_p$-complemented subspace $X$ of $L_p$ which is $K_p$-isomorphic to $L_p$ and such that $T\big|_X$ is a  $K_p-$isomorphism and $T(X)$ is $K_p$ complemented in $L_p$.
\end{thm}

Here we are mostly interested in the following corollary of Theorem~\ref{DJS}.

\begin{thm} \label{oldA}
Assume $1 < p < 2$. Then every non-Enflo operator on $L_p$ is narrow.
\end{thm}

It is an interesting problem whether Theorem~\ref{oldA} can be strengthened  to an analogue of Theorem~\ref{Ros}  for $1 < p \leq 2$.

\begin{prob} \label{pr:1}
Suppose $1 < p \leq 2$, and an operator $T:L_p\to L_p$ is such that for every  measurable set $A \subseteq[0,1]$ the restriction $T \bigl|_{L_p(A)}$ is not an isomorphic embedding. Does it follow that $T$ is narrow?
\end{prob}

In general this problem remains open, but in Section~\ref{secGnarrow} we obtain a partial answer under a suitable restriction. Namely, we consider operators $T$ so that for every $\varepsilon>0$ and every  measurable set $A \subseteq[0,1]$ there exists a function $x\in L_p$ with $\supp x\subseteq A$ and certain prescribed estimates for the distribution of $x$, so that $\|Tx\|<\varepsilon\|x\|$. These prescribed estimates are much less restrictive than the requirement that $x^2$ is the characteristic function of $A$, but they are not as general as the condition in Problem~\ref{pr:1}. To make this precise we introduce the following definitions.

\begin{defn} \label{d:mtrunc}
For any $x \in L_0$ and $M > 0$ we define the $M$-truncation $x^M$ of $x$ by setting
$$
x^M (t) = \left\{
            \begin{array}{ll}
              x(t), & \hbox{if} \,\,\, \bigl| x(t) \bigr| \leq M, \\
              M \cdot {\rm sign} \, \bigl( x(t) \bigr), & \hbox{if} \,\,\, \bigl| x(t) \bigr| > M.
            \end{array}
          \right.
$$
\end{defn}

\begin{defn} \label{d:acept}
Let $1 < p \leq 2$. A decreasing function $\varphi: (0, +\infty) \to [0,1]$ is said to be $p$-gentle if
$$
\lim\limits_{M \to +\infty} M^{2-p} \bigl( \varphi(M) \bigr)^p = 0.
$$
\end{defn}

\begin{defn} \label{d:dnar}
Let $1 < p \leq 2$, and let $X$ be a Banach space. We say that an operator $T \in \mathcal L(L_p,X)$ is gentle-narrow if there exists a $p$-gentle function $\varphi: (0, +\infty) \to [0,1]$ such that for every $\varepsilon > 0$, every $M > 0$ and every $A \in \Sigma$ there exists $x \in L_p(A)$ such that the following conditions hold
\begin{enumerate}
  \item[(i)] $\|x\| = \mu(A)^{1/p}$;
  \item[(ii)] $\|x - x^M\| \leq \varphi (M) \, \mu(A)^{1/p}$;
  \item[(iii)] $\|Tx\| \leq \varepsilon$.
\end{enumerate}
\end{defn}

Observe that every narrow operator is gentle-narrow with
$$
\varphi(M) = \left\{
               \begin{array}{ll}
                 1 - M, & \, \hbox{if} \,\, 0 \leq M < 1, \\
                 0, & \hbox{if} \,\, M \geq 1.
               \end{array}
             \right.
$$

Indeed, for every sign $x$ on $A$ one has that $\|x - x^M\| = \varphi(M) \, \mu(A)^{1/p}$ for each $M \geq 0$ where $\varphi$ is the above defined function.

Another condition  yielding that an operator $T \in \mathcal L(L_p,X)$ is gentle narrow is the following one: for each $A \in \Sigma$ and each $\varepsilon > 0$ there exists a mean zero gaussian random variable $x \in L_p(A)$ with the distribution
$$
d_x \stackrel{\rm def}{=} \mu \bigl\{ x < a \bigr\} = \frac{\mu(A)}{\sqrt{2 \pi \sigma^2}} \int_{- \infty}^a e^{- \frac{t^2}{2 \sigma^2}} dt
$$
and such that $\|Tx\| < \varepsilon$. One can show that in this case $T$ is gentle-narrow with $\varphi(M) = C e^{- \frac{M^2}{2 \sigma^2}}$, where $C$ is a constant independent of $M$.

The  main theorem of Section~\ref{secGnarrow} gives an affirmative answer to a weak version of Problem~\ref{pr:1}. Namely, we prove the following result.

\vspace{0,2 cm}
\noindent{\textbf{Theorem C.} \textit{Let} $1 < p \leq 2$. \textit{Then every gentle-narrow operator} $T :L_p\to L_p$
\textit{is narrow.}

\vspace{0,2 cm}

We finish the introduction with some remarks on the novelty of the approach here. The main results are all of the form that, given an operator with domain $L_p$
which sends functions of a certain kind to vectors of small norm, there exist also signs; i.e., functions taking only the values $\pm 1$ and sometimes also $0$, which
are sent by $T$ to vectors of relatively small norm. Naturally, the proofs are concentrated on building special signs out of given functions. There are two methods
that we use here to construct such signs and they may be the main contribution of this paper. One, in the proofs of theorems A and B, is probabilistic in nature and uses the martingale
structure of the partial sums of the Haar system, stoping times and the central limit theorem (although the notion of martingales and stoping times is not specifically
mentioned in the actual proofs). A similar idea originated in \cite{KS92} and was also used in \cite{KKW}. The other, in the proof of Theorem C, has geometrical nature and uses fine estimates on $L_p$ norms of functions of the form $1+y$
with $y$ of mean zero and bounded in absolute value by 1 and other inequalities for the $L_p$ norm.
%%%%%%%%%%%%%%%%%%%%%%%%%%%%%%%%%%%%%%%%%%%%%%%%%%%%%%%%%%%%%%%%%%%%%%%%%%%%%%%%%%%%%%%%%%%%%%%%%%%%%%%%%

\section{Preliminary results}
\label{prelres}

By a \textit{sign} we mean an element $x \in L_\infty$ which takes values in the set $\{-1,0,1\}$, and a {\it sign on a set $A \in \Sigma$} is any sign $x$ with ${\rm supp} \, x = A$. We say that a sign $x$ is of mean zero provided that
$\displaystyle{\int_{[0,1]} x \, d \mu = 0}$ holds.

An operator $T \in \mathcal L(L_p,X)$ is said to be {\it narrow} if
for each $A\in \Sigma$ and each $\varepsilon > 0$ there is a mean zero sign $x$ on $A$ such that $\|Tx\| < \varepsilon$. Equivalently, we can remove the condition on the sign $x$ to be of mean zero in this definition \cite[p.~54]{PP}.
The first systematic study of narrow operators was done in \cite{PP} (1990), however some results on them were known earlier. For more information on narrow operators we refer the reader to a recent survey \cite{P}.

We also consider a weaker notion.

\begin{defn} \label{d:somewhatnar}
An operator $T \in \mathcal L (L_p,X)$ is called \textit{somewhat narrow} if for each $A \in \Sigma$ and each $\varepsilon > 0$ there exists a set $B \in \Sigma, \,\, B \subseteq A$ and a sign $x$ on $B$ such that $\|Tx\| < \varepsilon \|x\|$.
\end{defn}

Obviously, each narrow operator is somewhat narrow. The inclusion embedding $J: L_p \to L_r$ with $1 \leq r < p < \infty$ is an example of a somewhat narrow operator which is not narrow.
However, for operators in $\mathcal L(L_p)$  for $1\le p\le 2$ and more generally for operators in $\mathcal L(L_p,X)$, with $X$ of type $p$ and $1\le p\le 2$ these two notions are equivalent.

Recall that $X$ is said to be of type $p$ if there is a constant $K$ such that for any unconditionally convergent series $\sum\limits_{n=1}^\infty x_n$ in $X$
     $$
      \left(\mathbb{E}_\theta\Bigl\| \sum\limits_{n=1}^\infty \theta_n x_n \Bigr\|^p\right)^{1/p} \leq K\left( \sum\limits_{n=1}^\infty \| x_n \|^p \right)^{1/p} $$
      where the expectation is taken over all sequences of signs $\theta=\{\theta_n\}_{n=1}^\infty$, $\theta_n=\pm1$. The fact that $L_p$, $1\le p\le 2$ is of type $p$ (with constant $K=1$) goes back probably to Orlicz and is easy to prove (see \eqref{type}).

\begin{thm} \label{A}
Let $1 \leq p \leq 2$ and let $X$ be of type $p$. Then every somewhat narrow operator $T \in \mathcal L(L_p,X)$ is narrow.
In particular any somewhat narrow operator $T \in \mathcal L(L_p)$ is narrow.
\end{thm}

Theorem~\ref{A} is not true for $p > 2$ as Example~\ref{R:Sch} shows.

\begin{proof}[Proof of Theorem \ref{A}] Let $K$ be the constant from the definition of the type $p$ property of $X$.
Fix any $A \in \Sigma^+$ and $\varepsilon > 0$. To prove that $T$ is narrow it is enough to prove that $\|Tx\| \leq \varepsilon \mu(A)^{1/p}$ for some sign $x$ on $A$.

Assume, for contradiction, that for each sign $x$ on $A$ one has that
$$\|Tx\| > \varepsilon \mu(A)^{1/p}.$$
We will construct a transfinite sequence $(A_\alpha)_{\alpha < \omega_1}$ of uncountable length $\omega_1$ of disjoint sets $A_\alpha \in \Sigma^+, \,\, A_\alpha \subset A$, which will give us the desired contradiction.

By the definition of a somewhat narrow operator, there exist a set $A_0 \in \Sigma^+$, $A_0 \subseteq A$ and a sign $x_0$ on $A_0$ such that
$$
\|Tx_0\| \leq \varepsilon K^{-1} \mu(A_0)^{1/p}.
$$

Observe that by our assumption, $x_0$ cannot be a sign on $A$, therefore, $\mu(A \setminus A_0) > 0$.

Suppose that for a given ordinal $0 < \beta < \omega_1$ we have constructed a transfinite sequence of disjoint sets $(A_\alpha)_{\alpha < \beta}\subseteq \Sigma^+, \,\, A_\alpha \subset A$ and a transfinite sequence $(x_\alpha)_{\alpha < \beta}$ of signs $x_\alpha$ on $A_\alpha$ such that
$$
\|Tx_\alpha\| \leq \varepsilon K^{-1} \mu(A_\alpha)^{1/p}.
$$

We set $B = \bigcup_{\alpha < \beta} A_\alpha$. Our goal is to prove that $\mu(A \setminus B) > 0$. Since $(x_\alpha)_{\alpha < \beta}$ is a disjoint sequence in $L_p(\mu)$ with $|x_\alpha| \leq 1$ a.e., one has that the series $\sum_{\alpha < \beta} x_\alpha$ unconditionally convergent, and so is the series $\sum_{\alpha < \beta} T x_\alpha$. By the definition of type $p$ there are sign numbers $\theta_\alpha = \pm 1, \,\, \alpha < \beta$ so that
\begin{equation} \label{eq:ineq}
\begin{split}
\Bigl\| \sum_{\alpha < \beta} \theta_\alpha T x_\alpha \Bigr\| &\leq K\Bigl(  \sum_{\alpha < \beta} \bigl\| T x_\alpha \bigr\|^p \Bigr)^{1/p} \leq \Bigl(  \sum_{\alpha < \beta} \varepsilon^p \mu \bigl( A_\alpha \bigr) \Bigr)^{1/p} = \varepsilon \mu(B)^{1/p}.
\end{split}
\end{equation}

Observe that $x = \sum_{\alpha < \beta} \theta_\alpha x_\alpha$ is a sign on $B$ and, by \eqref{eq:ineq},  $\|Tx\| \leq \varepsilon \mu(B)^{1/p}$. By our assumption, $x$ cannot be a sign on $A$ and hence $\mu(A \setminus B) > 0$. Using the definition of  a somewhat narrow operator, there exists  $A_\beta \in \Sigma^+, \,\, A_\beta \subseteq A$ and a sign $x_\beta$ on $A_\beta$ such that
$$
\|Tx_\beta\| \leq \varepsilon K^{-1}\mu(A_\beta)^{1/p}.
$$

Thus, the recursive construction is done.
\end{proof}

\section{A proof that every operator $T \in \mathcal L(L_p,\ell_r)$ is narrow for $2 < p, r < \infty$}
\label{seceverynarrow}

Throughout this section we use the following notation.

\begin{itemize}
  \item $\bigl(\overline{h}_n\bigr)_{n=1}^\infty$ -- the $L_\infty$-normalized Haar system;
  \vspace{0,2 cm}
  \item $(h_n)$ and $(h_n^*)$ -- the $L_p$- and $L_q$-normalized Haar functions respectively,
  where $1/p + 1/q = 1$.
  \end{itemize}

\begin{prop}\label{le1:p,r>1}
Suppose $1 \leq p< \infty$, $X$ is a Banach space with a basis $(x_n)$, $T \in \mathcal L(L_p,X)$ satisfies $\|Tx\|\geq 2 \delta$ for each mean zero sign $x \in L_p$
on $[0,1]$ and some $\delta>0$. Then for each $\varepsilon>0$ there exist an operator $S \in \mathcal L(L_p, X)$, a normalized block basis $(u_n)$ of $(x_n)$ and real numbers $(a_n)$ such that
\begin{enumerate}
  \item $S h_n = a_n u_n$ for each $n \in \mathbb N$ with $a_1 = 0$;
  \item $\|Sx\|\geq \delta$, for each mean zero sign $x \in L_p$ on $[0,1]$;
\item there exists a linear isometry $V$ of $L_p$ into $L_p$, which sends signs to signs, so that $\|Sx\|\leq \|TVx\| + \varepsilon$ for every $x \in L_p$ with $\|x\|=1$;
    \item there are finite codimensional subspaces $X_n$ of $L_p$ such that $\|Sx\|\leq \|TVx\| + 1/n$ for every $x \in X_n$ with $\|x\|=1$;
\end{enumerate}
If, moreover, $\|Tx\|\geq 2 \delta \|x\|$ for every sign $x$, then
$|a_n|\geq\delta$  for each $n \geq 2$.
\end{prop}

\begin{proof}
Let $(P_n)_{n=1}^\infty$ be the basis projections in $X$ with respect to the basis $(x_n)$ and $P_0 = 0$. First we construct an operator $\widetilde S$ which has all the desired properties of $S$, with the small difference that $\widetilde S$ is defined on the closed linear span $H$ of a sequence which is isometrically equivalent to the Haar system and the required properties hold with $L_p$ replaced with $H$. For this purpose, we construct a sequence of integers $0 = s_1 < s_2 < \ldots$, a tree ${(A_{m,k})_{m=0}^\infty \,}_{k=1}^{2^m}$ of measurable sets $A_{m,k}\subseteq [0,1]$ and an operator $\widetilde{S} \in \mathcal L \bigl( L_p(\Sigma_1), X \bigr)$, where $\Sigma_1$ is the sub-$\sigma$-algebra of $\Sigma$ generated by the $A_{m,k}$'s with the following properties:

\begin{enumerate}
  \item[(P1)] $A_{0,1} = [0,1]$;
  \item[(P2)] $A_{m,k} = A_{m+1, 2k-1} \sqcup A_{m+1, 2k}, \,\,\,
\mu\bigl( A_{m,k} \bigr) = 2^{-m}$ for all $m,k$;
 \item[(P3)] $\|\widetilde{S} x\| \geq \delta$ for each mean zero sign $x
\in L_p(\Sigma_1)$ on $[0,1]$;
  \item[(P4)] if $h_1'=\textbf{1}$ and $h_{2^m+k}' = 2^{m/p}
\bigl(\textbf{1}_{A_{m+1, 2k-1}} - \textbf{1}_{A_{m+1, 2k}}\bigr)$ for all $m,k$, then we have $\widetilde{S} h_1'=0$ and $\widetilde{S} h_n' = (P_{s_n} - P_{s_{n-1}}) T h_n'$ for $n = 2, 3, \ldots$;
\item[(P5)] for all $n$
$$\Bigl\|Th'_n - \bigl( P_{s_n} - P_{s_{n-1}} \bigr) Th'_n \Bigr\|\leq \frac{\varepsilon}{2^n}.
$$
\end{enumerate}

We will use the following convention: once the sets $A_{m+1, 2k-1}$ and $A_{m+1, 2k}$ are defined for given $m,k$, we consider $h_{2^m+k}'$ to be defined by the equality from (P4).
To construct a family with the above properties, we set $A_{1,1}=[0,1/2)$ and $A_{1,2}=[1/2,1]$. Then choose $s_2 > 1$ so that
$$
\bigl\|Th'_2 - \bigl(P_{s_2} - P_{s_1} \bigr) T h'_2 \bigr\| = \bigl\|Th'_2 - P_{s_2} T h'_2 \bigr\| \leq \frac{\varepsilon}{4}. \leqno (C_2)
$$

Since the operator $P_{s_2} T$ is finite rank and hence narrow, there exists a mean zero sign $x_{1,1}$ on the set $A_{1,1}$ such that $\displaystyle{\bigl\| P_{s_2} T x_{1,1} \bigr\|\leq \frac{\varepsilon}{16\cdot 2^{1/p}}}$. Then set $A_{2,1} = x_{1,1}^{-1}(1) = \bigl\{t\in
[0,1]: \,\, x_{1,1}(t)=1 \bigr\}$, $A_{2,2} = x_{1,1}^{-1}(1)$ and observe that $h'_3 = 2^{1/p} x_{1,1}$ and $\bigl\|P_{s_2} T h'_3\bigr\| \leq \varepsilon/16$. Now we choose $s_3 > s_2$ so that $\bigl\| T h_3' - P_{s_3} T h_3'\bigr\| \leq \varepsilon/16$. Then we have
$$
\bigl\|Th'_3 - \bigl(P_{s_3} - P_{s_2}\bigr) T h'_3 \bigr\|\leq
\frac{\varepsilon}{8}. \leqno (C_3)
$$

Since $P_{s_3} T$ is narrow, there exists a mean zero sign $x_{1,2}$ on $A_{1,2}$ such that $\displaystyle{\bigl\|P_{s_3} T x_{1,2}\bigr\| \leq \frac{\varepsilon}{32\cdot 2^{1/p}}}$. Put $A_{2,3} = x_{1,2}^{-1}(1)$ and $A_{2,4} = x_{1,2}^{-1}(-1)$. Observe that $h'_4 = 2^{1/p} x_{1,2}$ and
$\bigl\|P_{s_3} T h'_4\bigr\| \leq \varepsilon/32$. Choose $s_4 > s_3$ so that $\bigl\| T h_4' - P_{s_4} T h_4'\bigr\| \leq \varepsilon/32$. Then
$$
\bigl\|Th'_4 - \bigl(P_{s_4} - P_{s_3}\bigr) T h'_4 \bigr\|\leq
\frac{\varepsilon}{16}. \leqno (C_4)
$$

Further we analogously find a mean zero sign $x_{2,1}$ on $A_{2,1}$ such that $\displaystyle{\bigl\| P_{s_4} T x_{2,1} \bigr\| \leq \frac{\varepsilon}{64\cdot 2^{2/p}}}$. Then putting $A_{3,1} = x_{2,1}^{-1}(1)$, $A_{3,2} = x_{2,1}(-1)$, we obtain $h'_5 = 2^{2/p} x_{2,1}$ and $\bigl\|P_{s_4} T h'_5\bigr\| \leq \varepsilon/64$. Now choose $s_5 > s_4$ so that $\bigl\| T h_5' - P_{s_5} T h_5'\bigr\| \leq \varepsilon/64$, and obtain
$$
\bigl\|Th'_5 - \bigl(P_{s_5} - P_{s_4}\bigr) T h'_5 \bigr\|\leq
\frac{\varepsilon}{32}. \leqno (C_5)
$$

Continuing the procedure, we construct a sequence of integers $0 = s_1 < s_2 < \ldots$, a tree ${(A_{m,k})_{m=0}^\infty \,}_{k=1}^{2^m}$ of measurable sets $A_{n,k}\subseteq [0,1]$ which satisfies conditions (P1) and (P2), for which we have
$$
\Bigl\|Th'_n - \bigl( P_{s_n} - P_{s_{n-1}} \bigr) Th'_n \Bigr\|\leq \frac{\varepsilon}{2^n}. \leqno (C_n)
$$

Note that property (P4) defines the operator $\widetilde{S}$ on the system $(h_n')$. We show that $\widetilde{S}$ could extended by linearity and continuity on $L_p(\Sigma_1)$. Let $x = \sum_{n=1}^N \beta_n h'_n \in L_p (\Sigma_1)$ with $\|x\|=1$. Note that $|\beta_n| \leq 2$, because the Haar system is a monotone basis in $L_p$. Then by $(C_n)$, we obtain
\begin{align} \label{Ttildeprelim}
\bigl\|\widetilde{S} x \bigr\| &= \Bigl\| \sum_{n=2}^N \beta_n (P_{s_n} - P_{s_{n-1}} ) T h_n' \Bigr\|\\
&\leq \Bigl\| \sum_{n=2}^N \beta_n T h_n' \Bigr\| + \Bigl\| \sum_{n=2}^N \beta_n \bigl( T h_n' - (P_{s_n} - P_{s_{n-1}} ) T h_n' \bigr) \Bigr\| \notag \\
&\leq \|Tx\| + \max\limits_{n \geq 2} |\beta_n| \sum_{n=2}^N \bigl\| T h_n' - (P_{s_n} - P_{s_{n-1}} ) T h_n' \bigr\| \notag\\
&\leq \|Tx\| + 2 \sum_{n=2}^N \frac{\varepsilon}{2^n} < \|Tx\| + \varepsilon \notag
\end{align}

This also proves that $\widetilde{S}$ is well defined and bounded. Moreover, a similar computation to \eqref{Ttildeprelim} shows that, for some sequence $\varepsilon_L\to 0$ as $L\to\infty$, each $x \in L_p (\Sigma_1)$ of norm one and of the form $x = \sum_{n=L}^N \beta_n h'_n$ one has
\begin{equation} \label{Ttilde}
\bigl\|\widetilde{S} x \bigr\| \leq \|Tx\| + \varepsilon_L.
\end{equation}

It remains to verify that $\widetilde{S}$ satisfies (P3). Let $x = \sum_{n=1}^\infty \beta_n h'_n \in L_p(\Sigma_1)$ be a mean zero sign on $[0,1]$. Using the inequality
$$
|\beta_n| = \Bigl| \int_{[0,1]} h_n'^* x \, d \mu \Bigr| \leq \int_{[0,1]} |h_n'^* | \, d \mu = \|h_n'^* \|_1 \leq \|h_n'^*\|_q = 1,
$$
we obtain
\begin{align*}
\bigl\|\widetilde{S} x \bigr\| &= \Bigl\| \sum_{n=2}^\infty \beta_n (P_{s_n} - P_{s_{n-1}} ) T h_n' \Bigr\|\\
&\geq \Bigl\| \sum_{n=2}^\infty \beta_n T h_n' \Bigr\| - \Bigl\| \sum_{n=2}^\infty \beta_n \bigl( T h_n' - (P_{s_n} - P_{s_{n-1}} ) T h_n' \bigr) \Bigr\|\\
&\geq \|Tx\|- \sum_{n=2}^\infty \bigl\| T h_n' - (P_{s_n} - P_{s_{n-1}} ) T h_n' \bigr\|\\
&\geq 2 \delta - \sum_{n=2}^\infty \frac{\varepsilon}{2^n} = 2 \delta - \varepsilon \ge \delta
\end{align*}
if $\varepsilon\le \delta$. Thus, the desired properties of $\widetilde{S}$ are proved.

Now we are ready to define $S$, $(a_n)$ and $(u_n)$. Let $V: L_p \to L_p(\Sigma_1)$ be the linear isometry extending the equality $V h_n = h_n'$ for all possible values of indices ($V$ exists because of (P1) and (P2)). Set $S = \widetilde{S} \circ V$. Then, by \eqref{Ttildeprelim},  $\|Sx\|\le\|TVx\| + \varepsilon$ for all $x$ of norm one varifying (3). Also, (4) follows from \eqref{Ttilde}. Set $a_n = \|\widetilde{S} h_n' \|$ for all $n \in \mathbb N$,  and $u_n = \|\widetilde{S} h_n'\|^{-1} \widetilde{S} h_n'$ if $\widetilde{S} h_n' \neq 0$, and $u_n = \|y_{s_n}\|^{-1} y_n$ if $\widetilde{S} h_n' = 0$. By (P3) and (P4), $S$ satisfies (1) and (2) (one has $a_1 = 0$ because $Sh_1' = S h_1 = \widetilde{S} V h_1 = \widetilde{S} V h_1' = 0$).

If, moreover, $\|Tx\|\geq 2 \delta \|x\|$ for every sign $x$, then $\|Th'_n\|\geq 2\delta$ for all $n \geq 2$, and by $(C_n)$ we have
\begin{align*}
|a_n| = \|\widetilde{S} h_n\| &= \bigl\| (P_{s_n} - P_{s_{n-1}}) T h_n' \bigr\| \\
&\geq \|T h_n'\| - \|T h_n' - (P_{s_n} - P_{s_{n-1}}) T h_n' \| \geq 2 \delta - \frac{\delta}{2^{n-1}}\geq \delta.
\end{align*}
\end{proof}

Now we are ready to prove Theorem A. Recall that the only previously unknown case is $2<p,r<\infty$. Before starting the formal proof we would like to explain the idea. If
$T\in \mathcal L(L_p,\ell_r)$ is not narrow then, for some $\delta>0$ there is an operator $S\in \mathcal L(L_p,\ell_r)$ satisfying (1)-(3) of Proposition~\ref{le1:p,r>1}. Denote by $r_m=\sum_{k=1}^{2^m}\overline{h}_{2^m+k}$, $m=0,1\dots$, the Rademacher functions. Note that they are sent by $S$ to vectors with disjoint supports in $\ell_r$. For a fixed (large) $C$ and (even much larger) $N$ put
$$f=\frac{C}{\sqrt{N}}\sum_{n=2}^{2^{N+1}} \overline{h}_n=\frac{C}{\sqrt{N}}\sum_{m=0}^N r_m$$
and notice that since $r>2$, $\|Sf\|\le\|S\|CN^{1/r-1/2}$ is arbitrarily small (if $N$ is large enough). Also,
$f$ is approximately $C$ times a Gaussian variable and so has norm bounded by a constant (depending on $p$ only) times $C$. If $f$ were a sign we would be done, reaching a contradiction, but of course it is not. To remedy this, for each $\omega\in [0,1]$, we start adding the summands forming $f(\omega)$, namely $\frac{C}{\sqrt{N}}\overline{h}_n(\omega)$, one by one stopping whenever we leave the interval $[-1,1]$. If $C$ is large enough, we do leave this interval for the vast majority of the $\omega$-s. Whenever we leave the interval, since we just stopped the summation, we get that the absolute value of the stopped sum is 1, to within an error of $C/\sqrt{N}$. We thus got a function which is almost a sign and which, by the unconditionality of the $\ell_r$ basis, is sent to a function of norm smaller than $\|Sf\|$ which was arbitrarily small, reaching a contradiction to (2) of Proposition~\ref{le1:p,r>1}.

\begin{proof}[Proof of Theorem A for $2<p,r<\infty$.]
Suppose that an operator $T \in \mathcal L(L_p,\ell_r)$ is not narrow. Without loss of generality we may assume that $\|Tx\|_r\geq 2 \delta$ for each mean zero sign $x \in L_p$ on $[0,1]$ and some $\delta>0$. By Proposition \ref{le1:p,r>1} for $X = \ell_r$ and $x_n = e_n$, the unit vector basis of $\ell_r$ (and $\varepsilon\le \|T\|$, say), there exists an operator $S \in \mathcal L(L_p, \ell_r)$ with $\|S\|\le 2\|T\|$, which satisfies  conditions (1)-(3) of Proposition~\ref{le1:p,r>1} ((4) will not be used here). Since every normalized block basis of $(e_n)$ is isometrically equivalent to $(e_n)$ itself (see \cite[Proposition~2.a.1]{LTI}), we may and do assume that $u_n = e_n$.

Let $C>0$ and $N \in \mathbb N$. We denote $I_m^k = \supp \overline{h}_{2^m+k} = \bigl[ \frac{k-1}{2^m},\frac{k}{2^m} \bigr)$. We will define several objects depending on $N$ and $C$, and in order not to complicate the notation, we omit the indices $N$ and $C$. We start with the function
$$f=\frac{C}{\sqrt{N}}\sum_{n=2}^{2^{N+1}} \overline{h}_n.$$

Since $S$ has a special form, we obtain that
\begin{equation}\label{estimateSf}
\begin{split}
\|Sf\|_r&=
\Big\|\frac{C}{\sqrt{N}}\sum_{n=2}^{2^{N+1}} S\overline{h}_n\Big\|_r =
\Big\|\frac{C}{\sqrt{N}}\sum_{m=0}^N S\Big(\sum_{k=1}^{2^m} \overline{h}_{2^m+k}\Big)\Big\|_r\\
&=\frac{C}{\sqrt{N}} \Big(\sum_{m=0}^{N} \bigl\|S\Big(\sum_{k=1}^{2^m} \overline{h}_{2^m+k}\Big) \bigr\|_r^r\  \Big)^{1/r}\\
&\le\frac{C}{\sqrt{N}} \|S\| \Big(\sum_{m=0}^N \big\|\sum_{k=1}^{2^m} \overline{h}_{2^m+k} \big\|_p^r \Big)^{1/r}\\
&\le 2C\|T\|(N+1)^{1/r}N^{-1/2}.
\end{split}
\end{equation}

Since $r>2$, for $N$  large enough, $\|Sf\|_r$ is as small as we want.

Our goal is to select a subset $J \subseteq \{2, \ldots, 2^{N+1}\}$ so that the element $g = \frac{C}{\sqrt{N}}\sum_{n\in J} \overline{h}_n$ is very close to a sign. This will prove that $S$ fails (2) of Proposition~\ref{le1:p,r>1}, since by the special form of $S$,
$\|Sg\|_r\le\|Sf\|_r$ which was very small.

To achieve this goal we use a technique similar to a stopping time for a martingale. A similar method was first used in \cite{KS92}.

 Set
$$A=\Bigl\{\omega \in [0,1] : \max\limits_{1 \leq j \le 2^{N+1}} \Big| \frac{C}{\sqrt{N}} \sum_{i=1}^j \overline{h}_i(\omega) \Big| > 1 \Bigr\},$$
and
\begin{equation*}
\tau(\omega)=\begin{cases}
\min \Big\{ j \le 2^{N+1}:  \Big|\frac{C}{\sqrt{N}} \sum_{i=1}^j \overline{h}_i(\omega) \Big| > 1 \Big\},\ \  \text{if}\ \ \omega\in A,\\
2^N+k,\ \  \text{if}\ \  \omega\not\in A\ \  \text{and}\ \  \omega \in I_{N}^{k}.
\end{cases}
\end{equation*}

Observe that if $\tau(\omega) = 2^m + k$ then $ \omega \in I_m^k$. Indeed, this is clear if $ \omega\not\in A$. If $ \omega\in A$, then $$\Big|\frac{C}{\sqrt{N}}\sum_{i=1}^{2^m+k-1} \overline{h}_i(\omega)\Big|\le 1,$$
so $\overline{h}_{2^m+k}(\omega)\ne 0$ and thus  $ \omega \in I_{m}^{k}$.

Further, if there exists $\omega \in I_m^k$ with $\tau(\omega) \geq 2^m+k$ then for every $\xi \in I_m^k$ we have $\tau(\xi) \geq 2^m+k$. Indeed, since $ \omega \in I_m^k$, for every $i < 2^m+k$ and every $\xi\in I_m^k$ we have $\overline{h}_i(\omega) = \overline{h}_i(\xi)$. Thus,
$$
\Big|\frac{C}{\sqrt{N}}\sum_{i=1}^{2^m+k-1} \overline{h}_i(\xi)\Big| = \Big|\frac{C}{\sqrt{N}}\sum_{i=1}^{2^m+k-1} \overline{h}_i(\omega)\Big| \le 1.
$$
Thus, $\tau(\xi) \ge 2^m+k$.

Define a set $J$:
\begin{equation*}
\begin{split}
J&=\Big\{j = 2^m+k \leq 2^{N+1}: \exists\omega\in I_m^k\ \ \text{with}\ \ \tau(\omega) \geq j \Big\}\\
&= \Big\{j = 2^m+k \leq 2^{N+1}: \forall \xi \in I_m^k\ \ \ \ \tau(\xi) \geq j \Big\}.
\end{split}
\end{equation*}

Let $g:[0,1]\to\mathbb{R}$ be defined as:
$$
g(\omega) = \frac{C}{\sqrt{N}} \sum_{j \leq \tau(\omega)} \overline{h}_j(\omega).
$$

Since $ \overline{h}_{2^m+k}(\omega) = 0$ for every $\omega \notin I_m^k$, we have
$$
g(\omega)=\frac{C}{\sqrt{N}}\sum_{\{j = 2^m+k \leq \tau(\omega)\ \text{and} \  \omega \in I_m^k\}} \overline{h}_j(\omega) = \frac{C}{\sqrt{N}} \sum_{j \in J} \overline{h}_j(\omega).
$$
Thus, by the form of $S$ and \eqref{estimateSf},
\begin{equation} \label{duh38}
\begin{split}
\|Sg\|_r&=\Big\|\frac{C}{\sqrt{N}}\sum_{j \in J} S\overline{h}_j \Big\|_r \\
&\le\|Sf\|_r\le 2C\|T\|(N+1)^{1/r}N^{-1/2}.
\end{split}
\end{equation}

By the definitions of $\tau(\omega)$ and $g(\omega)$, for every $\omega\in A$ one has
$$
1<|g(\omega)|<1+\frac{C}{\sqrt{N}},
$$
and for every $\omega\in [0,1]\setminus A$, $g(\omega)\ne 0$, if $N$ is odd, being a non zero multiple of a sum of an odd number of $\pm 1$.

Define
$$\widetilde{g}(\omega)=\sgn(g(\omega)).$$
We have
$$\|g-\widetilde{g}\|_p\le \Big\|\frac{C}{\sqrt{N}}\mathbf{1}_A +\mathbf{1}_{[0,1]\setminus A}\Big\|_p.$$
Note that, by the Central Limit Theorem, for large $N$ we have
\begin{equation*}
\begin{split}
\mu([0,1]\setminus A)&=\mu\Big\{\omega: \Big|\frac{1}{\sqrt{N}}\sum_{m=1}^N\sum_{k=1}^{2^m} \overline{h}_{2^m+k}(\omega)\Big|\le \frac1C\Big\}\\
&\approx \frac{1}{\sqrt{2\pi}}\int_{-\frac1C}^{\frac1C} e^{-\frac{\omega^2}{2}} d\omega < \frac{1}{2C}.
\end{split}
\end{equation*}
Thus, for large $N$,
$$\big\|\mathbf{1}_{[0,1]\setminus A}\big\|_p\le\Big(\frac{1}{2C}\Big)^{1/p}.$$
Hence
$$\|g-\widetilde{g}\|_p\le \frac{C}{\sqrt{N}}+\Big(\frac{1}{2C}\Big)^{1/p}.$$
Thus, by \eqref{duh38},
\begin{equation*}
\begin{split}
\|S\widetilde{g}\|_r&\le\|Sg\|_r+\|S\|\|g-\widetilde{g}\|_p\\
&\le 2C\|T\|(N+1)^{1/r}N^{-1/2}+2\|T\|\Big(\frac{C}{\sqrt{N}}+\Big(\frac{1}{2C}\Big)^{1/p}\Big).
\end{split}
\end{equation*}
Since $r>2$, for every $\delta>0$, there exists $C>0$ and $N$ odd and large enough so that
\begin{equation*}
\|S\widetilde{g}\|_r<\delta.
\end{equation*}

It remains to observe that $\widetilde{g}$ is a mean zero sign on $[0,1]$. Indeed, the support of $\widetilde{g}$ is equal to $[0,1]$, since $N$ is odd. Observe that for every $\omega \in [0,1]$ and every $2^m+k \leq 2^{N+1}$,
$$
\overline{h}_{2^m+k}(\omega) = - \overline{h}_{2^{m}+(2^m-k)+1}(1-\omega).
$$
Thus, ${g}(\omega)=-{g}(1-\omega)$ for every $\omega\in[0,1]$, and
$$\mu\big(\big\{\omega\in[0,1]: {g}(\omega)>0\big\}\big)=\mu\big(\big\{\omega\in[0,1]: {g}(\omega)<0\big\}\big).$$
Thus, $\widetilde{g}$ is a mean zero  sign on $[0,1]$, and (2) of  Proposition~\ref{le1:p,r>1} fails.
\end{proof}

\section{$\ell_2$-strictly singular operators - Proof of Theorem~B} \label{secstrsng}

As we indicated in the introduction, the idea of the proof of Theorem~B  is very similar to that of the proof of Theorem~A. Suppose that $T:L_p\to X$ is $\ell_2$-strictly singular and not narrow. As in the proof of Theorem~A, we use Proposition~\ref{le1:p,r>1} to conclude that there exists  an operator $S$ of the particularly simple form which cannot be narrow,
quantified using a specific $\delta>0$. The $\ell_2$-strict singularity of $T$ and condition (4) of  Proposition~\ref{le1:p,r>1} guarantee that $S$ is $\ell_2$-strictly singular as well. This implies that there exists a function $f$ of the form
$$
f = \sum_{m=1}^N b_m r_m = \sum_{m=1}^N \sum_{k=1}^{2^m} b_m\overline{h}_{2^m+k},
$$
where $(r_n)$ is the Rademacher system, so that $\|Sx\|_X $ is arbitrarily small. Similarly as in the proof of
Theorem~A,  we construct a subset $J \subseteq \{2, \ldots 2^{N+1} \}$ so that the function $g = \sum_{2^m+k \in J} b_m \overline{h}_{2^m+k}$ is very close to a sign and $\|Sg\|_X \le \|Sf\|_X < \delta$, which gives us the desired contradiction.

\begin{proof}[Proof of Theorem~B]
Let $1<p<\infty$ and let $X$ be a Banach space with an unconditional basis $\{x_n\}_{n=1}^\infty$. Without loss of generality we assume that the suppression unconditionality constant of the basis is $1$; i.e., $\|\sum_{i\in\sigma'}a_ix_i\|\le \|\sum_{i\in\sigma}a_ix_i\|$ for all $\sigma'\subset\sigma$ and all coefficients $\{a_i\}_{i\in\sigma}$. Let $T:L_p\to X$ be an $\ell_2$-strictly singular operator and suppose the operator $T$ is not narrow. As in the proof of Theorem~A we may assume without loss of generality that $\|Tx\|_X\geq 2 \delta$ for each mean zero sign $x \in L_p$ on $[0,1]$ and some $\delta>0$. Therefore, there exist  operators $S \in \mathcal L(L_p, X)$ and $V\in \mathcal L(L_p)$ which satisfy conditions (1)-(4) of Proposition~\ref{le1:p,r>1} (for $\varepsilon=\|T\|$, say).

Since $T$ is $\ell_2$-strictly singular, $TV$ is $\ell_2$-strictly singular as well and it then easily follows from (4) of Proposition~\ref{le1:p,r>1} that $S$ is also $\ell_2$-strictly singular. It follows that for every $\varepsilon>0$, $C>0$ there exists
$$
f = \sum_{m=1}^N b_m r_m,
$$
so that
\begin{equation}\label{bmain}
\|f\|_p=1, \  \|Sf\|_X<\frac{\delta}{2C} \ \  \text{and}\ \ \max_{n\le N} |b_n|\le \varepsilon.
\end{equation}
Indeed,
fix an $M\in\mathbb N$ to be determined momentarily, pick disjoint $\sigma_k\subset \mathbb N$ each of size $M$, $k=1,2,\dots$,
and put $f_k= M^{-1/2}\sum_{n\in \sigma_k} r_n$. Note that the sequence $\{f_k\}$ is equivalent to an orthonormal basis so the subspace $H$ spanned by the $f_k$-s is isomorphic to a Hilbert space and so an $f$ satisfying the first two requirements in
\eqref{bmain} can be chosen as $f=\sum_{k=1}^K a_kf_k$. Since $\|f\|=1$, the $a_k$ are uniformly bounded (by a constant depending only on $p$) and if $M$ is large enough then the coefficients of $f$ with respect to the Rademacher system, $a_k/\sqrt{M}$ are smaller than $\varepsilon$.

As in the proof of Theorem~A we set
$$
A = \Bigl\{\omega \in [0,1] : \max\limits_{1 \leq 2^m+k \leq 2^{N+1}} \Big| C \sum_{2^n+i=2}^{2^m+k} b_n \overline{h}_{2^n+i}(\omega) \Big| > 1 \Bigr\},
$$

\begin{equation*}
\tau(\omega) = \begin{cases}
\min \Big\{2^m+k \leq 2^{N+1}: \,\, \Big|C \sum\limits_{2^n+i=2}^{2^m+k} b_n \overline{h}_{2^n+i}(\omega) \Big| > 1 \Big\},\ \  \text{if}\ \ \omega \in A,\\
2^N+k,\ \  \text{if}\ \  \omega\not\in A\ \  \text{and}\ \  \omega \in I_{N}^{k},
\end{cases}
\end{equation*}

\begin{equation*}
\begin{split}
J = \Big\{j = 2^m+k \leq 2^{N+1}: \exists\omega\in I_m^k\ \ \text{with}\ \ \tau(\omega) \geq j \Big\}
\end{split}
\end{equation*}
and
$$
g(\omega) = C \sum_{2^m+k \leq \tau(\omega)} b_m \overline{h}_{2^m+k}(\omega) = C \sum_{2^m+k \in J} b_m \overline{h}_{2^m+k}(\omega).
$$

Thus, by the fact that S sends the Haar functions to functions with disjoint support with respect to the basis $\{x_i\}$, by the $1$-unconditionality of this basis and by \eqref{bmain},
\begin{equation*}
\|Sg\|_X\le C\|Sf\|_X<\frac\delta2.
\end{equation*}

Note that as $\varepsilon\downarrow 0$, by the Central Limit Theorem with the Lindeberg condition (see e.g. \cite[Theorem~27.2]{Bill}), $f$ converges to a Gaussian random variable so, if $\varepsilon$ is small enough (independently of $C$),
\begin{equation*}
\begin{split}
\mu([0,1]\setminus A) & \le \mu \Big\{\omega: \big|\sum_{n=1}^N\sum_{i=1}^{2^n}b_n \overline{h}_{2^n+i}(\omega)\big|\le \frac1C\Big\}\approx \frac{1}{\sqrt{2\pi}}\int_{-\frac1C}^{\frac1C} e^{-\frac{\omega^2}{2}} d\omega\le \frac{1}{2C}.
\end{split}
\end{equation*}

Let $[0,1]\setminus A=A_1\sqcup A_2$, where $\mu(A_1)=\mu(A_2)$, and define
\begin{equation*}
\widetilde{g}(\omega)=
\begin{cases} \sgn(g(\omega))\ \  &{\rm if}\ \ \omega\in A\\
1\ \  \ \ &{\rm if}\ \ \omega\in A_1\\
-1\ \  &{\rm if}\ \ \omega\in A_2.
\end{cases}
\end{equation*}

Then  $\widetilde{g}$ is a mean zero sign on $[0,1]$ and
 for every $\omega\in A$,
$$1<|g(\omega)|<1+C\varepsilon.$$
Thus
$$\|g-\widetilde{g}\|_p\le C\varepsilon+\Big(1+\frac1C\Big)\Big(\frac{1}{2C}\Big)^{1/p},$$
and
\begin{equation*}
\begin{split}
\|S(\widetilde{g})\|_X&\le\|S(g)\|_X+\|S\|\|g-\widetilde{g}\|_p\\
&\le \frac\delta2+2\|T\|\Big(C\varepsilon+\Big(1+\frac1C\Big)\Big(\frac{1}{2C}\Big)^{1/p}\Big).
\end{split}
\end{equation*}
Hence there exist $C>0$ and $\varepsilon$  so that
\begin{equation*}
\|S\widetilde{g}\|_X<\delta,
\end{equation*}
which contradicts our assumption about $S$ and hence $T$.
\end{proof}

\section{Banach spaces $X$ so that every operator $T \in \mathcal L (L_p,X)$ is narrow.}

In this section we discuss the following question:

 \textit{For what Banach spaces $X$, is every operator $T \in \mathcal L (L_p,X)$  narrow?}

Following \cite{KP92}, we denote the class of such spaces by $\mathcal M_p$. As was shown in \cite{KP92}, the following spaces belong to these classes:

\begin{enumerate}
  \item \textit{$c_0 \in \mathcal M_p$ for any $1 \leq p < \infty$};
  \item \textit{$L_r \in \mathcal M_p$ if $1 \leq p < 2$ and $p < r$.}
\end{enumerate}

Moreover, the result in (2) is sharp: $L_r \notin \mathcal M_p$ if $p \geq 2$, or $1 \leq p < 2$ and $p \geq r$, as Example~\ref{R:Sch} shows.

An easy argument (see \cite[p.~63]{PP}) implies that

\begin{enumerate}
  \item[(3)] \textit{$\ell_r \in \mathcal M_p$ for any $1 \leq p < \infty$ and $1 \leq r < 2$.}
\end{enumerate}

Indeed, assume that $T \in \mathcal L(L_p, \ell_r)$ and $A \in \Sigma$. Consider a Rademacher system $(r_n)$ on $L_p(A)$. Since $[r_n]$ is isomorphic to $\ell_2$, by Pitt's theorem, the restriction $T|_{[r_n]}$ is compact, and hence, $\lim_{n \to \infty} \|Tr_n\| = 0$.

Theorem A asserts that

\begin{enumerate}
  \item[(4)] \textit{$\ell_r \in \mathcal M_p$ for $p,r \in (2, \infty)$.}
\end{enumerate}

Note thet the above results (1), (3) for $1<p<\infty$, and (4) also follow from Theorem B. Indeed, Theorem B implies that

\begin{enumerate}
  \item[(5)] \textit{if $X$ has an unconditional basis and does not contain an isomorphic copy of $\ell_2$, then $X \in \mathcal M_p$ for $p\in (1, \infty)$.}
\end{enumerate}

The inclusion embedding operator $J_{p,2}$ from $L_p$ to $L_2$ for $p \geq 2$ is non-narrow by definition, and so is its composition $S \circ J_{p,2}$ with an isomorphism $S: L_2 \to \ell_2$. Thus, we have

\begin{enumerate}
  \item[(6)] \textit{$\ell_2 \notin \mathcal M_p$ for $p \geq 2$.}
\end{enumerate}

However, in spite of the existence of  non-narrow operators from $L_p$ to $\ell_2$ if $p > 2$, all these operators must be small in the following sense.

\begin{defn} \label{dsignemb}
We say that an operator $T \in \mathcal L \bigl(L_p,X \bigr)$ is a \textit{sign embedding} if $\|Tx\| \geq \delta \|x\|$ for some $\delta > 0$ and every sign $x \in L_p$.
\end{defn}
Note that in some papers (see \cite{Ros81/82}, \cite{Ros83}) H.~P.~Rosenthal studied the notion of a sign embedding defined on $L_1$, but in his definition an additional assumption of injectivity of $T$ was required. Formally, this is not the same, and there is an operator on $L_1$ that is bounded from below at signs and is not injective (and even has a kernel isomorphic to $L_1$), see \cite{MP2}. But  if an operator on $L_1$ is bounded from below at signs as in Definition~\ref{dsignemb}, then there exists $A \in \Sigma^+$ such that the restriction $T \bigr|_{L_1(A)}$ is injective, and hence is a sign embedding in the sense of Rosenthal (cf. \cite{MP2}).

\begin{prop} \label{easyprop2}
Let $2 < p < \infty$. Then there does not exist an operator $T \in \mathcal L(L_p, \ell_2)$ which is a  sign embedding.
\end{prop}

\begin{proof}
Suppose on the contrary, that $T \in \mathcal L (L_p, \ell_2)$ and $\|Tx\| \geq 2\delta \|x\|$ for some $\delta > 0$ and each sign $x \in L_p$. Then by Proposition \ref{le1:p,r>1}, there exists an operator $S:L_p\to l_2$ which satisfies conditions $(1)$ and $(2)$ from Proposition \ref{le1:p,r>1}. Moreover,
$|a_{n,k}|\geq\delta$ for each $n=0,1,...$ and $k = 1, \ldots, 2^n$.

Let $x_n=\sum\limits_{k=1}^{2^n}2^{-n/p}h_{n,k}$. Note that
$x_n$ is a mean zero sign on $[0,1]$. Then
$Sx_n=\sum\limits_{k=1}^{2^n} 2^{-\frac np} a_{n,k}e_{n,k}$. Now we
have
$$
\|Sx_n\|\geq 2^{-\frac np}\delta 2^{\frac n2}=\delta 2^{n(\frac12 -
\frac 1p)}.
$$
Thus $\lim\limits_{n\to\infty}\|Sx_n\|=\infty$, which contradicts the boundedness of $S$.
\end{proof}

We summarize the above results.

\begin{remark} \label{finalfact}
$\,$
\begin{itemize}
  \item For every $1 \leq p < \infty$ every operator $T \in \mathcal L(L_p, c_0)$ is narrow.
  \item If $1 \leq p < 2$ and $p < r$ then every operator $T \in \mathcal L(L_p, L_r)$ is narrow, and this statement is not true for any other values of $p$ and $r$.
\item For   $p,r \in [1,2)\cup (2, +\infty)$, every operator $T \in \mathcal L(L_p, \ell_r)$ is narrow.
\item For $1<p< \infty$ and $X$ a Banch space with an unconditional basis which does not contain an isomorphic copy of $\ell_2$,   every operator $T :L_p\to X$ is narrow.
  \item For $2 < p < \infty$ there exists a non-narrow operator in $\mathcal L(L_p, \ell_2)$. However, there is no sign-embedding in $\mathcal L(L_p, \ell_2)$.
\end{itemize}
\end{remark}

\section{Gentle-narrow operators} \label{secGnarrow}

In this section we establish a weak sufficient condition for an operator $T \in \mathcal L(L_p)$
 to be narrow. More precisely, we prove Theorem C.
 For the proof, we need a few lemmas. First of them asserts that, in the definition of a gentle-narrow operator, in addition to properties (i)-(iii) we can claim one more property.

\begin{lemma} \label{pass}
Suppose $1 < p \leq 2$, $X$ is a Banach space and $T \in \mathcal L(L_p,X)$ is a gentle-narrow operator with a gentle function $\varphi: [0, +\infty) \to [0,1]$. Then for every $\varepsilon > 0$, every $M > 0$ and every $A \in \Sigma$ there exists $x \in L_p(A)$ such that the following conditions hold
\begin{enumerate}
  \item[(i)] $\|x\| = \mu(A)^{1/p}$;
  \item[(ii)] $\|x - x^M\| \leq \varphi (M) \, \mu(A)^{1/p}$;
  \item[(iii)] $\|Tx\| \leq \varepsilon$;
  \item[(iv)] $\displaystyle{\int_{[0,1]} x \, d \mu = 0}$.
\end{enumerate}
\end{lemma}

\begin{proof}[Proof of Lemma \ref{pass}]
Without loss of generality we assume that $\|T\| = 1$. Fix $\varepsilon > 0$, $M > 0$ and $A \in \Sigma$. Choose $n \in \mathbb N$ so that $\bigl( \mu(A)/n \bigr)^{1/p} < \varepsilon/4$ and decompose $A = A_1 \sqcup \ldots \sqcup A_n$ with $A_k \in \Sigma$ and $\mu(A_k) = \mu(A)/n$ for every $k = 1, \ldots, n$. Using the definition of a gentle-narrow operator, for each $k = 1, \ldots, n$ we choose $x_k \in L_p(A_k)$ so that $\|x_k\|^p = \mu(A)/n$, $\bigl\| x_k - x_k^M \bigr\| \leq \varphi(M) \, \mu(A)^{1/p}$, and $\|Tx_k\| < \varepsilon/(2n)$.

Without loss of generality, we may and do assume that
$$
\delta = \Bigl| \int_{[0,1]} x_n \, d \mu \Bigr| \geq \Bigl| \int_{[0,1]} x_k \, d \mu \Bigr|
$$
for $k = 1, \ldots, n-1$ (otherwise we rearrange $A_1, \ldots, A_n$). Observe that
\begin{equation*} %\label{march2}
\|x_k\| = \Bigl( \frac{\mu(A)}{n} \Bigr)^{1/p} < \frac{\varepsilon}{4} \, .
\end{equation*}

We choose inductively sign numbers $\theta_1 = 1$ and $\theta_2, \ldots, \theta_{n-1} \in \{-1,1\}$ so that for each $k = 1, \ldots, n - 1$ we have
$$
\Bigl| \int_{[0,1]} \sum\limits_{i=1}^k \theta_i x_i \, d \mu \Bigr| \leq \delta .
$$

Pick a sign $r$ on $A_n$ so that
\begin{equation} \label{march1}
\int_{[0,1]} r x_n \, d \mu = - \int_{[0,1]} \sum\limits_{i=1}^{n-1} \theta_i x_i \, d \mu
\end{equation}
(This is possible, because
$$
\Bigl| \int_{[0,1]} \sum\limits_{i=1}^{n-1} \theta_i x_i \, d \mu \Bigr| \leq \delta = \Bigl| \int_{[0,1]} x_n \, d \mu \Bigr| \,\,\, \Bigr).
$$

We set $x = \sum_{i=1}^{n-1} \theta_k x_k + r x_n$ and show that $x$ satisfies the desired properties.

(i) $\displaystyle{\,\,\, \|x\|^p = \sum\limits_{k=1}^n \|x_k\|^p = n \cdot \frac{\mu(A)}{n} = \mu(A)}$.

(ii) $\displaystyle{\|x - x^M\|^p = \sum\limits_{k = 1}^n \|x_k - x_k^M\|^p \leq n \cdot \bigl( \varphi(M) \bigr)^p \cdot \frac{\mu(A)}{n} = \bigl( \varphi(M) \bigr)^p \mu(A)}$.

(iii) Since $\|Tx_k\| < \varepsilon/(2n)$, for $k\le n$, $\|T\| = 1$ and by the definition of $x$ we get
\begin{align*}
\|Tx\| \leq \sum\limits_{k=1}^{n-1} \|Tx_k\| + \|T r x_n\| \leq \sum\limits_{k=1}^n \|Tx_k\| + \|x_n - r x_n\| < \frac{\varepsilon}{2} + \|2x_n\| < \varepsilon.
\end{align*}

Property (iv) for $x$ follows from \eqref{march1}.
\end{proof}

\begin{lemma} \label{le:march3}
Assume $1 < p \leq 2$, $a > 0$ and $|b| \leq a$. Then
\begin{equation} \label{kon}
(a+b)^p - p \, a^{p-1} b \geq a^p + \frac{p \, (p-1)}{2^{3-p}} \cdot \frac{b^2}{a^{2-p}} \, .
\end{equation}
\end{lemma}

\begin{proof}[Proof of Lemma \ref{le:march3}]
Dividing the inequality by $a^p$ and denoting $t = b/a$, we pass to an equivalent inequality
$$
f(t) \stackrel{\rm def}{=} (1 + t)^p - p \, t - 1 - \frac{p \, (p-1)}{2^{3-p}} \, t^2 \geq 0
$$
for each $t \in [-1,1]$, which we have to prove. Observe that
$$
f'(t) = p \, (1 + t)^{p-1} - p - \frac{p \, (p-1)}{2^{2-p}} \, t \,\,\,\,\, \mbox{and} \,\,\,\,\, f''(t) = \frac{p \, (p-1)}{(1+t)^{2-p}} - \frac{p \, (p-1)}{2^{2-p}} \, .
$$

Since $f''(t) > 0$ for every $t \in (-1,1)$ and $f'(0) = 0$, we have that $t_0 = 0$ is the point of a global minimum of $f(t)$ on $[-1,1]$. Since $f(0) = 0$, the inequality \eqref{kon} follows.
\end{proof}

\begin{lemma} \label{doubt}
Assume $1 < p \leq 2$, $A \in \Sigma$, $a \neq 0$, $y \in L_\infty(A)$, $|y(t)| \leq |a|$ for each $t \in A$, and $\displaystyle{\int_{[0,1]} y \, d \mu = 0}$. Then
$$
\bigl\| a \mathbf{1}_A + y \bigr\|_p^p \geq |a|^p \mu(A) + \frac{p \, (p-1)}{2^{3-p}} \cdot \frac{\|y\|_2^2}{|a|^{2-p}} \, .
$$
\end{lemma}

\begin{proof}[Proof of Lemma \ref{doubt}]
Evidently, it is enough to consider the case when $a > 0$ which one can prove by integrating inequality \eqref{kon} written for $b = y(t)$.
\end{proof}

\begin{lemma} \label{random}
Let $1 < p \leq 2$, let $T \in \mathcal L (L_p)$ be a gentle-narrow operator with a gentle function $\varphi: [0, +\infty) \to [0,1]$ and $\|T\| = 1$. Then for all $M > 0$,   $\delta > 0$,   $B \in \Sigma^+$,   $\eta \in (0, 1/2)$, and   $y \in L_p$ with $\eta \leq |y(t)| \leq 1 - \eta$ for all $t \in B$, there exists $h \in B_{L_\infty(B)}$ such that:

\begin{enumerate}
  \item $\displaystyle{\|Th\| < 2 \, \mu(B)^{1/p} \frac{\varphi(M)}{M}}$;
  \item $\displaystyle{\|y \pm \eta h\|^p > \|y\|^p + \frac{p \, (p-1)}{2^{3-p}} \cdot \frac{\eta^2}{(1 - \eta)^{2-p}} \cdot \mu(B) \cdot \frac{\bigl( 1 - \varphi(M) \bigr)^2}{M^2} - \delta.}$
\end{enumerate}
\end{lemma}

\begin{proof}[Proof of Lemma \ref{random}]
We fix $M$, $\delta$, $B$, $\eta$ and $y$ as in the assumptions of the Lemma. Since we need to prove strict inequalities, we may and do assume without loss of generality that $y$ is a simple function on $B$
\begin{equation} \label{eq:petya1}
y \cdot \textbf{1}_B = \sum\limits_{k=1}^m b_k \textbf{1}_{B_k}, \,\,\,\,\, B = B_1 \sqcup \cdots \sqcup B_m, \,\,\,\,\, \eta \leq |b_k| \leq 1 - \eta.
\end{equation}

For each $k = 1, \ldots, m$ we choose $x_k \in L_p(B_k)$ so that
\begin{enumerate}
  \item[(i)] $\displaystyle{\|x_k\|^p = \mu(B_k)}$;
  \vspace{0,2 cm}
  \item[(ii)] $\displaystyle{\|x_k - x_k^{M}\| \leq \varphi (M) \, \mu(B_k)^{1/p}}$;
  \vspace{0,2 cm}
  \item[(iii)] $\displaystyle{\|Tx_k\| \leq \frac{\varphi(M) \, \mu(B)^{1/p}}{m}}$;
  \item[(iv)] $\displaystyle{\int_{[0,1]} x_k \, d \mu = 0}$.
\end{enumerate}

We set $x = \sum\limits_{k=1}^m x_k$ and $h = M^{-1} x^M$, and show that $h$ has the desired properties.

(1). Observe that (iii) implies that
\begin{equation} \label{eq:bh1}
\|Tx\| \leq \sum\limits_{k=1}^m \|Tx_k\| < m \frac{\varphi(M) \, \mu(B)^{1/p}}{m} = \varphi(M) \, \mu(B)^{1/p},
\end{equation}
and (ii) yields
\begin{equation} \label{eq:bh2}
\|x - x^M\|^p = \sum\limits_{k=1}^m \|x_k - x_k^M\|^p \leq \sum\limits_{k=1}^m \bigl( \varphi (M) \bigr)^p \mu(B_k) = \bigl( \varphi (M) \bigr)^p \mu(B).
\end{equation}

Combining \eqref{eq:bh1} and \eqref{eq:bh2}, we get
$$
\|Th\| = \frac{\|T (x^M)\|}{M} \leq \frac{\|Tx\|}{M} + \frac{\|x - x^M\|}{M} < 2 \, \mu(B)^{1/p} \frac{\varphi(M)}{M} \, .
$$

(2). Using the well known inequality for norms in $L_p$ and $L_2$ (see \cite[p.~73]{Car}), (i) and \eqref{eq:bh2} we obtain
\begin{align*}
\|x^M\|_2 &\geq \|x^M\|_p \, \mu(B)^{1/2 - 1/p} \geq \bigl( \|x\| - \|x - x^M\| \bigr) \mu(B)^{1/2 - 1/p}\\
&\geq \mu(B)^{1/p} \bigl( 1 - \varphi(M) \bigr) \, \mu(B)^{1/2 - 1/p} = \bigl( 1 - \varphi(M) \bigr) \, \mu(B)^{1/2},
\end{align*}
and hence,
\begin{equation} \label{eq:bh3}
\|x^M\|_2^2 \geq \bigl( 1 - \varphi(M) \bigr)^2 \, \mu(B).
\end{equation}

Thus,
\begin{align*}
&\|y \pm \eta h\|^p = \bigl\| y \cdot \textbf{1}_{[0,1] \setminus B} \bigr\|^p + \sum\limits_{k=1}^m \bigl\| b_k \cdot \textbf{1}_{B_k} \pm \eta M^{-1} x_k^M \bigr\|\\
&\stackrel{\mbox{\tiny by Lemma \ref{doubt}}}{\geq} \bigl\| y \cdot \textbf{1}_{[0,1] \setminus B} \bigr\|^p + \sum\limits_{k=1}^m |b_k|^p \mu(B_k) + \frac{p \, (p-1)}{2^{3-p}} \cdot \frac{\eta^2}{M^2} \sum\limits_{k=1}^m \frac{\|x_k^M\|_2^2}{|b_k|^{2-p}}
\end{align*}

(using \eqref{eq:bh3}, $|b_k| \leq 1 - \eta$ and the equality $\sum\limits_{k=1}^m \|x_k^M\|_2^2 = \|x^M\|_2^2$)
$$
\geq \|y\|^p + \frac{p \, (p-1)}{2^{3-p}} \cdot \frac{\eta^2}{(1 - \eta)^{2-p}} \cdot \mu(B) \cdot \frac{\bigl( 1 - \varphi(M) \bigr)^2}{M^2} \, .
$$

Note that the reason for including $\delta$ in (2) is to make possible the reduction to simple functions in the proof.
\end{proof}

\begin{proof}[Proof of Theorem C]
Let $T \in \mathcal L (L_p)$ be a gentle-narrow operator with a $p$-gentle function $\varphi: [0, +\infty) \to [0,1]$. By Theorem~\ref{A}, to prove that $T$ is narrow,   it is enough to prove that it is somewhat narrow. Without loss of generality, we may and do assume that $\|T\| = 1$. Fix any $A \in \Sigma^+$ and $\varepsilon > 0$, and prove that there exists a sign $x \in L_p(A)$ such that $\|Tx\| < \varepsilon \|x\|$. Consider the set
$$
K_\varepsilon = \Bigl\{ y \in B_{L_\infty(A)}: \,\,\, \|Ty\| \leq \varepsilon \|y\| \Bigr\}.
$$

By arbitrariness of $\varepsilon$, it is enough to prove the following statement:
\begin{equation}\label{A.star}
\bigl(\forall \varepsilon_1 > 0 \bigr) \bigl(\exists \,\,\, \mbox{a sign} \,\,\, x \in L_p(A) \bigr) \bigl( \exists y \in K_\varepsilon \bigr): \,\,\,\,\, \|x - y \| < \varepsilon_1 \|y\|.
\end{equation}

Indeed, if \eqref{A.star} is true for each $\varepsilon$ and $\varepsilon_1$, we choose a sign $x \in L_p(A)$ and $y \in K_{\varepsilon/2}$ such that $\|x - y\| < \varepsilon_1 \|y\|$ where
\begin{equation} \label{eq:eps1}
\varepsilon_1 = \frac{\varepsilon}{2 \varepsilon + 2}.
\end{equation}

Since $\varepsilon_1 < 1$, the inequality $\|y\| \leq \|x\| + \|x - y\| < \|x\| + \varepsilon_1 \|y\|$ implies
$$
\|y\| < \frac{1}{1 - \varepsilon_1} \|x\|,
$$
and hence,
\begin{equation*}
\begin{split}
\|Tx\| &\leq \|Ty\| + \|x - y\| < \frac{\varepsilon}{2} \, \|y\| + \varepsilon_1 \|y\|\\
&= \Bigl( \frac{\varepsilon}{2} + \varepsilon_1 \Bigr) \|y\| < \frac{1}{1 - \varepsilon_1} \Bigl( \frac{\varepsilon}{2} + \varepsilon_1 \Bigr) \|y\| {\buildrel {{\small \rm{by \, \eqref{eq:eps1}}}}\over =} \varepsilon \|x\|.
\end{split}
\end{equation*}

To prove \eqref{A.star}, suppose for contradiction  that \eqref{A.star}  is false. Thus there exists $\varepsilon_1 > 0$ so that
\begin{equation} \label{eq:contrary}
\bigl(\forall \,\,\, \mbox{sign} \,\,\, x \in L_p(A) \bigr) \bigl( \forall y \in K_\varepsilon \bigr): \,\,\,\,\, \|x - y \| \geq \varepsilon_1 \|y\|.
\end{equation}

Denote $\lambda = \sup  \{ \|y\|:   y \in K_\varepsilon  \}$. Notice that since $T$ is gentle-narrow, $\lambda > 0$. Set
\begin{equation} \label{eq:eta}
\eta = \frac{\varepsilon_1 \lambda}{4^{1/p} \mu(A)^{1/p}}
\end{equation}
and observe that
\begin{equation} \label{eq:etain}
\eta < \frac{\varepsilon_1 \mu(A)^{1/p}}{4^{1/p} \mu(A)^{1/p}} < \frac{1}{2} \, .
\end{equation}

Using that $\varphi$ is $p$-gentle, we choose $M > 0$ and $\delta_1 > 0$ so that
\begin{equation} \label{eq:emmmm}
\bigl( 1 - \varphi (M) \bigr)^2 \geq 1/2
\end{equation}
and
\begin{equation} \label{eq:emm1}
M^{2-p} \bigl( \varphi (M) \bigr)^p \leq \varepsilon^p \, \frac{p \, (p-1)}{16} \, \left( \frac{\eta}{1 - \eta} \right)^{2-p} - \delta_1 \, \frac{M^2 \varepsilon^p (1 - 2^p \eta^p)}{\eta^p 2^{2p-2} \varepsilon_1^p \lambda^p} \, .
\end{equation}

Pick $\varepsilon_2 > 0$ and then $\delta_2 > 0$ so that
\begin{equation} \label{eq:epstwo1}
(\lambda - \varepsilon_2)^p + \frac{p \, (p-1)}{2^{6-2p} M^2} \cdot \left( \frac{\eta}{1 - \eta} \right)^{2-p} \cdot \frac{\varepsilon_1^p \lambda^p}{1 - 2^p \eta^p} - \delta_2 > \lambda^p
\end{equation}
(by \eqref{eq:etain} the second summand in the left-hand side of the inequality is positive). Setting $\delta = \min \{ \delta_1, \delta_2\}$, by \eqref{eq:emm1} and \eqref{eq:epstwo1} we obtain
\begin{equation} \label{eq:emm}
M^{2-p} \bigl( \varphi (M) \bigr)^p \leq \varepsilon^p \, \frac{p \, (p-1)}{16} \, \left( \frac{\eta}{1 - \eta} \right)^{2-p} - \delta \, \frac{M^2 \varepsilon^p (1 - 2^p \eta^p)}{\eta^p 2^{2p-2} \varepsilon_1^p \lambda^p} \, .
\end{equation}
and
\begin{equation} \label{eq:epstwo}
(\lambda - \varepsilon_2)^p + \frac{p \, (p-1)}{2^{6-2p} M^2} \cdot \left( \frac{\eta}{1 - \eta} \right)^{2-p} \cdot \frac{\varepsilon_1^p \lambda^p}{1 - 2^p \eta^p} - \delta > \lambda^p.
\end{equation}

We choose $y \in K_\varepsilon$ with
\begin{equation} \label{eq:vaj}
\|y\| > \max \Bigl\{ \frac{\lambda}{2^{1/p}}, \, \lambda - \varepsilon_2 \Bigr\} .
\end{equation}

Define a sign $x$ in $L_p(A)$ by
$$
x(t) = \left\{
         \begin{array}{ll}
           0, & \hbox{if} \,\, |y(t)| \leq 1/2,\\
           {\rm sign} \, (y), & \hbox{if} \,\, |y(t)| > 1/2
         \end{array}
       \right.
$$
and put $$B = \bigl\{ t \in A: \,\, \eta \leq |y(t)| \leq 1 - \eta \bigr\}.$$

Since   $y \in K_\varepsilon$, it follows from \eqref{eq:contrary} that:
\begin{align*}
\varepsilon_1^p \|y\|^p &\leq \|x - y\|^p = \int_B |x - y|^p d \mu + \int_{A \setminus B} |x - y|^p d \mu\\
&\leq \frac{\mu(B)}{2^p} + \bigl( \mu(A) - \mu(B) \bigr) \eta^p.
\end{align*}

Hence, using \eqref{eq:vaj} and \eqref{eq:eta}, we deduce that
$$
\mu(B) \Bigl( \frac{1}{2^p} - \eta^p \Bigr) \geq \varepsilon_1^p \|y\|^p - \mu(A) \eta^p \geq \varepsilon_1^p \, \frac{\lambda^p}{2} - \frac{\varepsilon_1^p \lambda^p}{4} \, = \varepsilon_1^p \, \frac{\lambda^p}{4},
$$
that is,
\begin{equation} \label{eq:muofB}
\mu(B) \geq \frac{2^{p-2} \varepsilon_1^p \lambda^p}{1 - 2^p \eta^p} \, .
\end{equation}

In particular, by \eqref{eq:etain} we have that $\mu(B) > 0$. Observe that \eqref{eq:emm} and \eqref{eq:muofB} imply
\begin{align} \label{eq:new}
M^{-p} \bigl( \varphi(M) \bigr)^p &\leq \varepsilon^p \, \frac{p \, (p-1)}{16 M^2} \, \left( \frac{\eta}{1 - \eta} \right)^{2-p} - \frac{\varepsilon^p \delta (1 - 2^p \eta^p)}{\eta^p 2^{2p - 2} \varepsilon_1^p \lambda^p} \notag\\
&\leq \varepsilon^p \, \frac{p \, (p-1)}{16 M^2} \, \left( \frac{\eta}{1 - \eta} \right)^{2-p} - \frac{\varepsilon^p \delta}{\eta^p 2^p \mu(B)} \, .
\end{align}

By Lemma \ref{random}, we pick $h \in B_{L_\infty(B)}$ so that (1) and (2) hold.
Now comes the first time that we use the fact that the range space is $L_p$, $1<p\le2$. Recall that for all $u,v\in L_p$,
\begin{equation}\label{type}
\rm{Ave}_\pm\|u\pm v\|^p\le \|u\|^p+\|v\|^p.
\end{equation}
Indeed, by the convexity of the function $t^{2/p}$,
$$
\int\rm{Ave}_\pm|u\pm v|^p\le \int(\rm{Ave}_\pm|u\pm v|^2)^{p/2}=
\int(u^2+v^2)^{p/2}\le \int|u|^p+|v|^p.
$$
So we can choose a sign number $\theta \in \{-1,1\}$ so that
$$
\|Ty + \theta \eta T h \|^p \leq \|Ty\|^p + \eta^p \|Th\|^p
$$
and set $z = y + \theta \eta h$. Then by (1) and  the choice of $y \in K_\varepsilon$,
\begin{equation} \label{eq:firstt}
\begin{split}
\|Tz\|^p &\leq \|Ty\|^p + \eta^p \|Th\|^p \leq \varepsilon^p \|y\|^p + \eta^p 2^p \mu(B) M^{-p} \bigl( \varphi(M) \bigr)^p\\
 &{\buildrel {{\small \eqref{eq:new}}}\over\leq} \varepsilon^p \|y\|^p + \eta^p 2^p \mu(B) \varepsilon^p \, \frac{p \, (p-1)}{16 M^2} \, \left( \frac{\eta}{1 - \eta} \right)^{2-p} - \varepsilon^p \delta.
\end{split}
\end{equation}

On the one hand, by condition (2) and \eqref{eq:emmmm}, we have
\begin{equation} \label{eq:onehand}
\|z\|^p \geq \|y\|^p + \frac{p \, (p-1)}{2^{4-p} M^2} \cdot \frac{\eta^2}{(1 - \eta)^{2-p}} \cdot \mu(B) - \delta.
\end{equation}

Then  \eqref{eq:firstt} together with \eqref{eq:onehand} give
$$
\frac{\|Tz\|^p}{\|z\|^p} \leq \varepsilon^p,
$$
and that yields $z \in K_\varepsilon$ (note that $z \in B_{L_\infty(A)}$ by definitions of $z$ and $B$). On the other hand, we can continue the estimate \eqref{eq:onehand} taking into account the choice of $y$, \eqref{eq:vaj} and \eqref{eq:epstwo}, as follows
\begin{align*}
\|z\|^p &> (\lambda - \varepsilon_2)^p + \frac{p \, (p-1)}{2^{4-p} M^2} \cdot \frac{\eta^2}{(1 - \eta)^{2-p}} \cdot \mu(B) - \delta\\
&\stackrel{\small{\rm by   \eqref{eq:muofB}}}{\geq}  (\lambda - \varepsilon_2)^p + \frac{p \, (p-1)}{2^{4-p} M^2} \cdot \frac{\eta^2}{(1 - \eta)^{2-p}} \cdot \frac{2^{p-2} \varepsilon_1^p \lambda^p}{1 - 2^p \eta^p} - \delta \,\, \stackrel{\small{\rm by \eqref{eq:epstwo}}}{>} \, \lambda^p.
\end{align*}

This contradicts the choice of $\lambda$.
\end{proof}

We remark that Example~\ref{R:Sch} demonstrates that there is no analogue of Theorem~C which is true for $p > 2$.

\begin{ack}
We thank Bill Johnson for valuable comments.
\end{ack}

% ------------------------------------------------------------------------

%\subsection*{Acknowledgment}
%Many thanks to our \TeX-pert for developing this class file.
% ------------------------------------------------------------------------

\end{document}